\documentclass[11pt]{article}
\usepackage{amsmath}
\usepackage{amstext,amsbsy}
\usepackage{epsfig,multicol}
\usepackage{amsthm}
\usepackage{amssymb,latexsym}
\usepackage{color}
\usepackage{tikz}
\usepackage{enumerate}
\usepackage{framed}
\usepackage{caption}
\usepackage{subcaption}
\usetikzlibrary{decorations.pathreplacing}
\usepackage[colorlinks,allcolors=blue]{hyperref}

\usepackage{xcolor}
\usepackage{mathdots}
\usepackage{yhmath}
\usepackage{booktabs}
\usepackage{blkarray}
\usepackage[all]{xy}
\usepackage{amsfonts}
\usepackage[english]{babel}
\usepackage{mathtools}
\usepackage[]{algorithm2e}
\usepackage{MnSymbol}
\usepackage{stackengine}

\newtheorem{thm}{Theorem}%[section]

\newtheorem{exa}[thm]{Example}
\newtheorem{lem}[thm]{Lemma}
\newtheorem{cor}[thm]{Corollary}

\newtheorem{rem}[thm]{Remark}

\newcommand{\supp}{\mathop{\mathrm{supp}}}
\newcommand{\partition}{\mathop{\mathrm{Part}}}
\newcommand{\des}{\mathop{\mathrm{des}}}
 
\makeatletter
\newenvironment{proofof}[1]{\par
  \pushQED{\qed}%
  \normalfont \topsep6\p@\@plus6\p@\relax
  \trivlist
  \item[\hskip\labelsep
        \bfseries
    Proof of #1\@addpunct{.}]\ignorespaces
}{%
  \popQED\endtrivlist\@endpefalse
}
\makeatother

\begin{document}
%opening
\title{Random shuffles on trees using extended promotion}
\author{
Svetlana Poznanovi\'c and Kara Stasikelis \\ [6pt]
Department of Mathematical Sciences\\
Clemson University, Clemson, SC 29634, USA\\[5pt]
}
\date{} 
%\date{March 2014} 
\maketitle
\begin{abstract} 
The Tsetlin library is a very well studied model for the way an arrangement of books on a library shelf evolves over time.  One of the most interesting properties of this Markov chain is that its spectrum can be computed exactly and that the eigenvalues are linear in the transition probabilities. In this paper we consider a generalization which can be interpreted as a self-organizing library in which the arrangements of books on each shelf are restricted to be linear extensions of a fixed poset. The moves on the books are given by the extended promotion operators of Ayyer, Klee, and Schilling while the shelves, bookcases, etc. evolve according to the move-to-back moves as in the the self-organizing library of Bj\"orner. We show that the eigenvalues of the transition matrix of this Markov chain are $\pm 1$ integer combinations of the transition probabilities if the posets that prescribe the restrictions on the book arrangements are rooted forests or more generally, if they consist of ordinal sums of a rooted forest and so called ladders. For some of the results we show that the monoids generated by the moves are either $\mathcal{R}$-trivial or, more generally, in $\textbf{DO(Ab)}$ and then we use the theory of left random walks on the minimal ideal of such monoids to find the eigenvalues. Moreover, in order to give a combinatorial description of the eigenvalues in the more general case, we relate the eigenvalues when the restrictions on the book arrangements change only by allowing for one additional transposition of two fixed books.
\end{abstract} 

 %\noindent{\bf Keywords: } linear extension, promotion, Markov chain, Tsetlin library, eigenvalues
 
%\noindent {\bf MSC Classification:} 60J10, 60C05, 05E99, 20M07, 20M30

{\renewcommand{\thefootnote}{} \footnote{\emph{E-mail addresses}:
spoznan@clemson.edu (S.~Poznanovi\'c), stasike@g.clemson.edu (K.~Stasikelis)}

%\footnotetext[1]{The first author is partially supported by NSF grant DMS-1312817. } 

%1 ------------------------------------------------------------
\section{Introduction}\label{S: Introduction}

The Tsetlin library is a well studied finite state Markov chain. The states are the permutations of $[n]$ representing $n!$ possible arrangements of $n$ books on a shelf connected via the move-to-back moves: the book $i$ is picked up with probability $x_{i}$  and put at the end of the shelf.  Hendricks~\cite{hendricks1972stationary, hendricks1973extension} found the stationary distribution, while the fact that the eigenvalues of the transition matrix have an elegant formula was discovered (independently) by Donnelly~\cite{donnelly1991heaps}, Kapoor and Reingold~\cite{kapoor1991stochastic}, and Phatarfod~\cite{phatarfod1991matrix}. This Markov chain has been generalized in different ways.  In this paper we consider a generalization of the Tsetlin library which combines the two models from Bj\"orner~\cite{bjorner2009note} and Ayyer et al.~\cite{ayyer2014combinatorial} and whenever possible we use the notation from these two papers.

Consider a rooted tree $T$ whose leaves are all at the same depth (distance from the root), $d$. Let $L$ denote the set of leaves of $T$ and $I$ denote the set of inner nodes (nodes that are not leaves). Suppose that at each inner node $v$ of depth $d-1$ a poset $P_v$ on the children is given; we refer to these as \emph{leaf posets}. A \emph{linear extension} of a poset $P$ is a total ordering $\pi = \pi_1 \pi_{2} \cdots \pi_n$ of its elements such that $\pi_i \prec \pi_j$ implies $i < j$. The set of linear extensions of $P$ is denoted by $\mathcal{L}(P)$. The set of total orderings of $T$ is \[\mathcal{L}(T) \cong \displaystyle \bigotimes_{v \in I} \mathcal{L}(P_v).\]

We will consider a Markov chain with state set $\mathcal{L}(T)$.  The operations are given by certain subsets of  $L$. Specifically, let \[\mathcal{A}(L) = \{ E \subseteq L : \text{ no two elements of } E \text{ are siblings}\}.\] A node $v \in T$ is $E$-related if some descendant of $v$ is contained in $E$. Let $C_v$ be the set of children of the node $v$ and  \[C_v^E = \{v \in C_v : v \text{ is } E\text{-related}\}.\] 

In order to explain the moves in our Markov chain we need to define two operations: extended promotion and pop shuffling. Consider a \emph{naturally labeled} poset $P$ on the set $[n]$, with partial order $\preceq$, where $P$ is naturally labeled if $i \prec j$ in $P$ implies $i < j$ as integers.

The extended promotion operator was introduced in~\cite{ayyer2014combinatorial}. It generalizes Sch{\"u}tzenberger's  promotion operator $\partial$~\cite{schutzenberger1972promotion}, which can be expressed in terms of more elementary operators $\tau_{i}$ as shown in~\cite{haiman1992dual, malvenuto1994evacuation}.  Namely, for $i=1, \ldots,n$ and $\pi=\pi_1 \cdots \pi_n \in \mathcal{L}(P)$, let 
 \[\tau_i \pi = \begin{cases} \pi_1 \cdots \pi_{i-1}\pi_{i+1} \pi_i \cdots \pi_n & \text{if } \pi_i \text{ and } \pi_{i+1} \text{ are incomparable in } P, \\ \pi & \text{otherwise}. \end{cases}\]
 In other words, $\tau_i$ acts nontrivially if the interchange of $\pi_i$ and $\pi_{i+1}$ yields a linear extension of $P$. The \emph{extended promotion operator} $\partial_i$, $1 \leq i \leq n$, on $\mathcal{L}(P)$ is defined by \[\partial_i = \tau_{n-1} \cdots \tau_{i+1}\tau_i\] and, in particular, $\partial_{1} = \partial$. Note that the operators act from the left; so  $\tau_i$ is applied first, then $\tau_{i+1}$, etc. The operator $\widehat{\partial}_{i}$ is defined in the following way:
 \[ \text{ for } \; \; \; \pi, \pi' \in \mathcal{L}(P), \; \; \;  \widehat{\partial}_{i} \pi = \pi' \; \; \; \text{ if and only if } \; \; \; \pi' = \partial_{\pi^{-1}(i)} \pi. \]
 
\begin{exa}
Let $P$ be the poset with Hasse diagram as given in Figure~\ref{firstposet}. Then the set of linear extensions of $P$ is $\mathcal{L}(P) = \{12345,12354, 21345, 21354\}$ and for $\pi = 21345$ we have $\widehat{\partial}_2 \pi = 12354$. 
\end{exa}
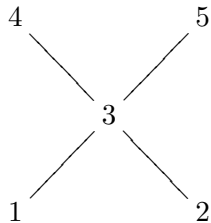
\begin{figure}[hb]
 \[\xymatrix{ 4\ar@{-}[dr] & & 5\ar@{-}[dl] \\ & 3 \ar@{-}[dr]\ar@{-}[dl] & \\ 1 & & 2 }\]
\caption{An example of a poset which is a ladder.}
\label{firstposet}
\end{figure}

The (elementary) pop shuffle is an operation on the elements of the symmetric group $S_n$~\cite{bidigare1999combinatorial, brown1998random}. Thinking again of books on a shelf, this operation models behavior when multiple readers are allowed to check out books before they are returned back on the shelf. The assumption is that after all readers have checked out their books (not all books need to be checked out by someone), the first reader places their books back at the end of the shelf in the order they were originally found. Then the second reader places their books at the end of the shelf in the order they were found. This continues until all readers have returned their books to the end of the shelf. A formal definition of a pop shuffle uses the language of ordered set partitions and is given in Section~\ref{background}.

\begin{exa}
Say 	there are four books labeled $1,2,3,4$ on one shelf and three readers $R_1, R_2, \text{ and } R_3$. Suppose the starting ordering of the books is $3142$. Say $R_1$ checks out books $2$ and $3$, $R_2$ checks out book $4$, and $R_3$ checks out book $1$. Then $R_1$ puts back books $2$ and $3$ in the order they were originally found, i.e., $32$. Then $R_2$ puts back book 4 at the end of the shelf, i.e., the arrangement after this  is $324$. Finally $R_3$ puts back book $1$ at the end of the shelf, so  the result of this pop shuffle is the arrangement $3241$. 
\end{exa}

Now, we define an action of an element $E \in \mathcal{A}(L)$ on $\mathcal{L}(T)$ with probability $x_E$ in the following way. Let $\pi = (\pi_v)_{v \in I}$ be a given total ordering of $T$. Then $\widehat{\partial}_E\pi = (\widehat{\partial}_{E_v}\pi_v)_{v\in I}$ where 
\begin{equation}
\label{operation}
\widehat{\partial}_{E_v}\pi_v = \begin{cases} \widehat{\partial}_{C_v^E} \pi_v & \text{ if depth}(v) = d-1 \\  \underline{\beta}^E_v  \pi_v & \text{ otherwise.} \end{cases}
\end{equation}
Here, $\underline{\beta}^E_v$ means pop shuffling, i.e., moving to back the $E$-related elements. So, in other words, in each move in our Markov chain the total ordering of $T$ is rearranged locally at each inner node so that the elements of $E$ are promoted and the $E$-related elements not in $E$ are moved to the back while their original order is preserved. 

\begin{exa}\label{ex: rootedtree1}
Let $T$ be as in Figure~\ref{treeex1}. 
Then $\mathcal{L}(T) = \{123,132, 312\} \times \{4\} \times \{56, 65\}$. Consider $\pi = (132, 4, 56) \in \mathcal{L}(T)$ and $E = \{1,4\}$, then $C_5^E = \{1\}, C_6^E = \{4\}$, and $C_7^E = \{5,6\}$. So, $\widehat{\partial}_{\{1,4\}}$ promotes $1$ and $4$ within the first two components of $\pi$. On the last component which is comprised of internal nodes, it acts as a pop shuffle. Specifically, since $5$ and $6$ are both $E$-related $\underline{\beta}^E_7$ moves to back both $5$ and $6$ while preserving their original order. Thus, $\widehat{\partial}_{\{1,4\}} \pi = (312, 4, 56).$ As an another illustration, $\widehat{\partial}_{\{1\}} \pi = (312, 4, 65)$ because only 1 and 5 are $\{1\}$-related.
\end{exa}

\begin{figure}[h]
\begin{center}
\scalebox{1}{$$ \xymatrix{  & & 7\ar@{-}[dr]\ar@{-}[dl]&  & & \\ & 5\ar@{-}[d]\ar@{-}[dr]\ar@{-}[dl]& & 6\ar@{-}[d]  & & P_5 = \hspace{-.3in}& 2\ar@{-}[d] & 3 & P_6 = \hspace{-.3in}& 4 \\ 1& 2& 3& 4 & &  & 1 & & & & }  $$} 
\end{center}
\caption{A tree $T$ of depth $2$ and its leaf posets.}
\label{treeex1}
\end{figure}
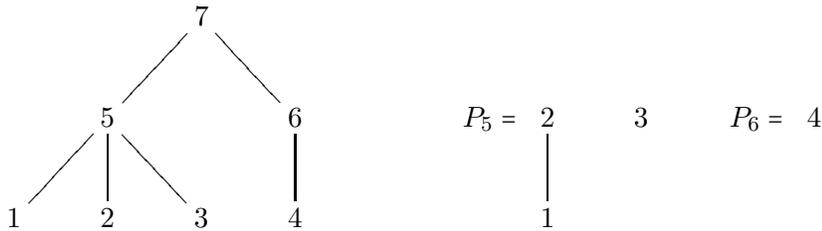

Let $M^{T_{P}}$ be the row stochastic transition matrix of the Markov chain described above with underlying tree $T$ and leaf poset $P$. Since there is no possibility for confusion, in the first three sections we suppress $P$ and write just $M^{T}$.  In this paper we describe the eigenvalues of $M^{T}$ when the leaf posets are rooted forests or consist of ordinal sums of a forest and so called ladders. The precise definitions of all the basic notions needed and the notation are given in Section~\ref{background}. Our main results show that the eigenvalues of $M^{T}$ in these two cases are $\pm 1$ integer combinations of the transition probabilities $x_{E}$. When the underlying tree $T$ is of depth 1, i.e., it has only a root and leaves, this Markov chain reduces to the extended promotion Markov chain given in~\cite{ayyer2014combinatorial}. When no leaf has a sibling, then we recover the move-to-front scheme on trees  in~\cite{bjorner2008random, bjorner2009note}. Based on the fact that both of these chains are irreducible when the probabilities $x_{E}$ are all non-zero, one can deduce that the Markov chain under consideration in this paper is also irreducible. In fact it can easily be seen that it is also aperiodic and thus converges to a unique stationary distribution.  Let $\mathcal{M}^{T}$ be the monoid generated by the transformations on $\mathcal{L}(T)$ induced by the operations $\widehat{\partial}_{E}$, $E \in \mathcal{A}(L)$. Note that sometimes even though $E \neq E'$, the induced transformations may be equal. The monoid $\mathcal{M}^{T}$ acts faithfully on $\mathcal{L}(T)$  and the Markov chain on linear extensions is equivalent to the left random walk on the minimal left ideal of $\mathcal{M}^{T}$ (see for example~\cite{ayyer2015markov}). The properties of this walk depend on the structure of the monoid $\mathcal{M}^{T}$. This was the fact that was used in~\cite{bjorner2009note} (where the monoid is a left-regular band) and~\cite{ayyer2014combinatorial} (where the monoid is $\mathcal{R}$-trivial) but the idea goes back to the seminal paper on random walks on left-regular bands of Brown~\cite{brown2000semigroups}. Recently, a unified framework to compute the stationary distribution of any finite irreducible Markov chain or equivalently of any irreducible random walk on a finite semigroup was developed in~\cite{Rhodes:2017rt}. 

In our case, in Section~\ref{maintheorems} we first show that when the leaf posets are rooted forests the monoid $\mathcal{M}^{T}$ is $\mathcal{R}$-trivial and we use this to find the eigenvalues of the transition matrix. Then we address the case when the leaf posets are ordinal sums of a forest and a ladder. In this case the monoid is no longer $\mathcal{R}$-trivial but we show that it belongs in a larger class, so called $\textbf{DO(Ab)}$. Then we use the recently developed theory for eigenvalues of left random walks for such monoids~\cite{ayyer2015markov, steinberg2006mobius, steinberg2008mobius} to find the eigenvalues. While this approach gives a combinatorial description of the eigenvalues in the case of rooted forests, in the latter case, the results are still expressed in terms of the associated monoid. For that reason, in Section~\ref{alternate} we use the approach we used in~\cite{poznanovic2017properties} to describe how one can compute the eigenvalues in the second case in a more combinatorial way knowing the eigenvalues in the case when the leaf posets are forests. 

Finally, one might ask what happens if we associate a poset to each inner node (and not only to the ones at depth $d-1$) and rearrange multiple books, shelves, etc. at a time while requesting that the result is a linear extension of the associated poset. This is in fact the question that started this work. It requires a generalization of the definition of a pop shuffle to linear extensions. It would be interesting to see if a nice generalization exists, but despite our efforts we could not find one.

\section{Background and Notation}
\label{background}

\subsection{Posets}
 
For a poset $P$, we say $y$ is a \emph{successor} of $x$ if $x \prec y$ and there is no $z$ such that $x \prec z \prec y$.  In this case we call $x \prec y$ a \emph{covering relation}.  A \emph{rooted tree} is a connected poset in which each vertex has at most one successor. A union of rooted trees is called a \emph{rooted forest}. An \emph{upset} (or \emph{upper set}) $S$ in a poset is a subset such that if $x \in S$ and $y \succeq x$, then $y \in S$. The upsets of a given poset form a lattice under inclusion. Consider a poset $P$ with minimal element $\hat{0}$ and maximal element $\hat{1}$; then for each element $x \in P$, the \emph{derangement number} of $x$~\cite{brown2000semigroups} is \begin{equation}\label{equation2} d_x = \displaystyle \sum_{y \preceq x} \mu(x,y) f([y, \hat{1}]),\end{equation} where $f([y, \hat{1}])$ is the number of maximal chains in the interval $[y, \hat{1}]$ and $\mu$ is the M{\"o}bius function~\cite{sta97} recursively defined by \[   \mu(x,y) = \begin{cases} 1 & \text{ if } x = y \\ - \displaystyle \sum_{x\preceq z \prec y} \mu(x,z) & \text{ if } x \prec y \\ 0 & \text{ otherwise.} \end{cases}   \]
 
 Let $P$ and $Q$ be two posets. The \emph{direct sum} of $P$ and $Q$ is the poset $P +Q$ on their disjoint union such that $x \preceq y$ in $P +Q$ if either (a) $x, y \in P$ and $x \preceq y$ in $P$ or (b) $x,y \in Q$ and $x \preceq y$ in $Q$. The \emph{ordinal sum} $P \oplus Q$ is a poset on their union such that:
\begin{enumerate}
\item For $x, y \in P$, $x \preceq y \in P \oplus Q$ if and only if $x \preceq y \in P$. 
\item For $x, y \in Q$, $x \preceq y \in P \oplus Q$ if and only if $x \preceq y \in Q$.  
\item For all $x \in P$ and $y \in Q$, $x \preceq y$ in $P \oplus Q$. 
\end{enumerate}
We will say that the poset $P$ is a \emph{ladder} of rank $k$ if $P = Q_1 \oplus \cdots \oplus Q_k$ where $Q_{i}$ is an antichain of size 1 or 2 for all $i= 1, \ldots, k$. The poset in Figure~\ref{firstposet} is a ladder of rank 3.

To formally define pop shuffles, we need the notion of set partitions. For any set $S$, a \emph{set partition} of $S$ is a set of disjoint nonempty subsets of $S$ whose union is $S$, we refer to each disjoint subset as a \emph{block}. The set of set partitions of $S$ is denoted by $\partition(S)$.  An \emph{ordered set partition} is a set partition with a linear ordering on the blocks. The set of ordered set partitions of $S$ is denoted by $\partition^{\text{ord}}(S)$. We will denote an ordered partition by $\underline{\beta}$ and its underlying set partition by $\beta$. 

$\partition(S)$ is a lattice ordered by reverse refinement. That is, $\alpha \leq \beta$ if and only if each block of the partition $\alpha$  is a union of blocks from $\beta$. In this case, $\beta$ is said to be a \emph{refinement} of $\alpha$. For two ordered set partitions, $\underline{\alpha}= (a_1, \ldots, a_\ell)$ and $\underline{\beta }= (b_1, \ldots, b_m)$, let  \[\underline{\alpha} \circ \underline{\beta} = (a_i \cap b_j),\] where the blocks are ordered by the indices $(i, j)$ in lexicographic order and empty blocks are omitted. With this operation $\partition^{\text{ord}}(S)$ is a monoid (semigroup with identity $(S)$).

Let $ \underline{\beta} = (B_1, \ldots, B_m)$ be an ordered partition of $[n]$ where each $B_i$ is a block. Then $\underline{\beta}$ acts on an element of $S_n$ by taking the elements in $B_1$ to the end of the permutation, while preserving the order in which they occur originally. The elements from $B_2$ are then placed after the elements from $B_1$, while preserving their original order. This is continued until the elements from $B_m$ are placed at the end of the permutation. Such a move $\underline{\beta}$ is called an \emph{elementary pop shuffle} and the action on a permutation $\pi$ is denoted by $\underline{\beta} \pi$. One can see that if $\pi = \pi_{1}\cdots \pi_{n}$, then $\underline{\beta} \pi = \beta \circ (\{\pi_{1} \}, \cdots, \{\pi_{n}\})$ and $(\underline{\alpha} \circ \underline{\beta}) \pi = \underline{\alpha}( \underline{\beta} \pi)$. Recall that  $C_v^E$ is the set of $E$-related children of the vertex $v$. With this notation the linear extension $\underline{\beta}_v^E \pi_{v}$ in~\eqref{operation} is exactly what we get when we take $\underline{\beta}_v^E = ( C_v \setminus C_v^E, C_v^E)$.

For a partition $\alpha =(\alpha_{v}) \in \partition(T \setminus L)$, we say that $E$ is $\alpha$-compatible if $\alpha_v$ is a refinement of $\beta^E_v$ for every $v \in T \setminus L$.

\subsection{$\mathcal{R}$-trivial monoids}

The left and right orders on a semigroup $\mathcal{M}$ were introduced by Green~\cite{green1951structure}. We follow the same convention as in~\cite{ayyer2014combinatorial}. 

Let $\mathcal{M}$ be a semigroup. For $x, y \in \mathcal{M}$, the left and right orders are defined by 
\begin{equation}
\begin{split}
x \leq_{\mathcal{R}} y\text{ if } y = xu \text{ for some } u \in \mathcal{M}, \\
x \leq_{\mathcal{L}} y\text{ if } y = ux \text{ for some } u \in \mathcal{M}. 
\end{split}
\end{equation}

\noindent A monoid $\mathcal{M}$ is said to be $\mathcal{R}$-trivial if $y \mathcal{M} = x \mathcal{M}$ implies $x = y$. 

A finite monoid $\mathcal{M}$ is said to be weakly ordered~\cite{schocker2008radical} if there is a finite upper semilattice $(L^{\mathcal{M}}, \preceq)$ together with two maps $\supp, \des \colon \mathcal{M} \rightarrow L^{\mathcal{M}}$ satisfying the following three axioms$\colon$
\begin{enumerate}
\item $\supp$ is a surjective monoid morphism, that is, $\supp(xy) = \supp(x) \vee \supp(y)$ for all $x, y \in \mathcal{M}$ and $\supp(\mathcal{M}) = L^{\mathcal{M}}$. 
\item If $x,y \in \mathcal{M}$ are such that $xy \leq_{\mathcal{R}}x$, then $\supp(y) \preceq \des(x)$. 
\item If $x, y \in \mathcal{M}$ are such that $\supp(y) \preceq \des(x)$, then $xy= x$. 
\end{enumerate}

\begin{thm}[\cite{berg2011primitive}]\label{ltrivial} Let $\mathcal{M}$ be a finite monoid. Then $\mathcal{M}$ is weakly ordered if and only if $\mathcal{M}$ is $\mathcal{R}$-trivial. 
\end{thm}

For an $\mathcal{R}$-trivial monoid $\mathcal{M}$, the associated semilattice $L^\mathcal{M}$ can be taken to be the set of left ideals generated by the idempotents in $\mathcal{M}$ ordered by reverse inclusion.

We will apply the following theorem for $\mathcal{R}$-trivial monoids to describe the eigenvalues of the transition matrix $M^{T_{P}}$ in the case when the leaf posets are rooted forests (Theorem~\ref{theorem: main1}). Let $\mathcal{C}$ be the set of chambers, that is, the set of maximal elements in the monoid $\mathcal{M}$ under $\geq_{\mathcal{R}}$. For $X \in L^{\mathcal{M}}$, define $c_X$ to be the number of chambers in $\mathcal{M}_{\geq X}$. This is precisely the number of $c \in \mathcal{C}$ such that $c \geq_{\mathcal{R}} x$, where $x \in \mathcal{M}$ is any fixed element such that $\supp(x) = X$.

\begin{thm}[\cite{ayyer2015markov}]\label{theorem: rmonoid}
Let $\{w_x\}$ be a probability distribution on $\mathcal{M}$, a finite $\mathcal{R}$-trivial monoid, that acts on the state space $\Omega$.  Let $M$ be the transition matrix for the random walk of $\mathcal{M}$ on $\Omega$ driven by the $w_x$'s. For each $X \in L^{\mathcal{M}}$ and $x$ such that $\supp(x) = X$, $M$ has an eigenvalue
\begin{equation}
\label{rev}
\lambda_X = \displaystyle \sum_{\supp(y) \preceq X} w_y
\end{equation}
\noindent with (possibly null) multiplicity given by 
\begin{equation}
\label{rmult}
m_X = \displaystyle \sum_{Y \succeq X} \mu(X, Y) c_Y,
\end{equation}
\noindent where $\mu$ is the M\"obius function of $L^{\mathcal{M}}$. These are all the eigenvalues of $M$.
\end{thm}

%%%%%%%%%%%%%%%%%%%%%%%%%%%
%%%%%%%%%%%%%%%%%%%%%%%%%%%
\subsection{The \textbf{DO(Ab)} class}

Let $x, y \in S$ for a semigroup $S$. Then  \[x \leq_{\mathcal{J}} y \ \ \text{ if } x= u y v \text{ for some } u, v \in S.\]

The elements $x$ and $y$ are in the same $\mathcal{J}$-class if $x \leq_{\mathcal{J}} y$ and $y \leq_{\mathcal{J}} x$. In particular, $x$ and $y$ are $\mathcal{J}$-equivalent if and only if $SxS = SyS$, i.e., if they generate the same two-sided ideal.  Let $\mathcal{P}(S)$ be the poset of $\mathcal{J}$-classes where $J \preceq_{\mathcal{J}} J'$ if $x \preceq_{\mathcal{J}} y$ for all (any) $x \in J$ and $y \in J'$. A $\mathcal{J}$-class is an \emph{orthodox semigroup} if the idempotents  form a subsemigroup. For a finite semigroup $S$ and an idempotent element $x \in S$, the \emph{maximal subgroup} is the group of units $(\{u :  \exists v, uv = vu = \text{id}\})$ of the submonoid $xSx$. The maximal subgroups depends only on the $\mathcal{J}$-class of $x$ up to isomorphism. A semigroup is \emph{regular} if for each element $x$ in $S$ there exists $y$ such that $yxy = y$. The class of \textbf{DO(Ab)} consists of all finite semigroups whose regular $\mathcal{J}$-classes are orthodox semigroups and whose maximal subgroups are abelian.

The following theorem of Steinberg~\cite{steinberg2006mobius} characterizes the monoids in \textbf{DO(Ab)} in terms of their representations. 

\begin{thm}[\cite{steinberg2006mobius}]\label{complexrep}
Let $S$ be a finite semigroup. Then the following are equivalent$\colon$
\begin{enumerate}
\item $S \in \mathrm{\mathbf{DO(Ab)}}$;
\item every irreducible complex representation of $S$ is a homomorphism $\phi: S \rightarrow \mathbb{C}$;
\item every complex representation of $S$ is equivalent to one by upper triangular matrices;
\item $S$ admits a faithful complex representation by upper triangular matrices.
\end{enumerate}
\end{thm}

The following theorem of Steinberg~\cite{steinberg2006mobius,steinberg2008mobius} gives an explicit representation of the eigenvalues for the left random walk on a minimal left ideal of a semigroup in the class \textbf{DO(Ab)}. 

\begin{thm}[\cite{steinberg2006mobius,steinberg2008mobius}]\label{theorem: doab}
Let $S \in \mathrm{\mathbf{DO(Ab)}}$ with generating set $X$ and let $L$ be a minimal left ideal. Assume that $S$ has left identity. Choose a maximal subgroup $H_J$, with identity $e_J$, for each regular $\mathcal{J}$-class $J$. Let $\{w_x\}_{x \in X}$ be a probability distribution on $X$. Then the transition matrix for the left random walk on $L$ can be placed in upper triangular form over $\mathbb{C}$. Moreover, there is an eigenvalue $\lambda_{(J, \chi)}$ for each regular $\mathcal{J}$-class $J$ and irreducible character $\chi$ of $H_J$ given by the formula
\begin{equation} \label{evdoab} \lambda_{(J, \chi)} = \displaystyle \sum_{x \in X, x \geq_\mathcal{J} J} w_x \cdot \chi(e_J x e_J) \end{equation}
with multiplicity 
\begin{equation} \label{mdoab} \frac{1}{|H_J|} \displaystyle \sum_{x \in H_J} \chi(x^{-1})  \sum_{\substack{  J' \in \mathcal{P}(S) \\J' \preceq J }} |\mathrm{Fix}_L(e_{J'} x e_{J'})| \mu(J', J) \end{equation}
where $\mathrm{Fix}_L(s)$ is the number of fixed points of $s$ acting on the left of $L$. Some of these multiplicities may be 0 but these are all the eigenvalues of the transition matrix.
\end{thm}

%2 ------------------------------------------------------------
\section{The eigenvalues of $M^{T_{P}}$: an algebraic treatment}
\label{maintheorems}

In this section we describe the eigenvalues of the transition matrix $M^{T}$ of the Markov chain from Section~\ref{S: Introduction} for certain classes of leaf posets. The first main result treats the case when the leaf posets are rooted forests. 

\begin{thm}\label{theorem: main1} 
Let $T$ be a rooted tree of depth $d$ with $k$ vertices at depth $d-1$: $v_{1}, \ldots, v_{k}$. Suppose all leaf posets are rooted forests and let $M^{T}$ be the transition matrix of the random walk on $\mathcal{L}(T)\colon$
\[M^{T}(\pi, \pi') = \sum_{E : \widehat{\partial}_E \pi = \pi'} x_E\]
\noindent for $\pi, \pi' \in \mathcal{L}(T)$. 
Then for an upset $S$ of $P$ and $\displaystyle \alpha \in \partition(T\setminus L)= \bigotimes_{\substack{v \in I \\ \mathrm{d}(v) \neq d-1}} \partition(C_{v}) $, $M^{T}$ has an eigenvalue 
\[ \lambda_{(S, \alpha)} = \displaystyle \sum_{\substack{E \subseteq S, \ E \in \mathcal{A}(L) \\ E \text{ is } \alpha\text{-compatible}}} x_E \]
\noindent with multiplicity $m_{(S, \alpha)} = d_{S_1} \cdots d_{S_k} m_\alpha,$
where  $d_{S_i}$ is the derangement number of $S_i = S \cap P_{v_i}$ in the lattice of upsets of $P_{v_i}$ and \begin{equation}\label{equation3} m_\alpha = \displaystyle \prod_{\substack{v \in I \\ \text{d}(v) \neq d-1}}  \prod_{B \in \alpha_v} (|B| -1)!.\end{equation} These are all the eigenvalues of $M^{T}$.
\end{thm}

The product in~\eqref{equation3} is over the blocks $B$ of the partitions $\alpha_{v}$.

%$d_S = d_{S_1} \cdots d_{S_k}$

\begin{exa}\label{ex: tree with poset structure}

Consider the tree $T$ with the leaf posets in Figure~\ref{treeex1}. The lattices of upsets of $P_5$ and $P_6$ are given in Figure~\ref{treeseigenvalues}. Using formula~\eqref{equation2} we get $d_{S_5} $ is $1$ if  $S_5 \in \{\emptyset, 2, 123\}$ and it is  0 otherwise. Similarly, $d_{S_6} $ is $1$  if  $S_6 = 4$ and  it is 0 otherwise. So, $d_{S_{5}}d_{S_{6}}$ is $1$ if  $S \in \{4, 24, 1234\}$ and 0 otherwise. Furthermore, ~\eqref{equation3} gives $m_{\{5,6\}} = 1$ and $m_{\{56\}} = 1$. Thus, the eigenvalues of the transition matrix  are 
%\begin{enumerate}
%\item[] $\lambda_{(1234,\{56\})} = x_{14} + x_{24} + x_{34}+ x_\emptyset$,
%\item[] $\lambda_{(1234, \{5,6\})} = x_{14} + x_{24} + x_{34} + x_1 + x_2 + x_3 + x_4+ x_\emptyset$, 
%\item[] $\lambda_{(24, \{56\})} = x_{24}+ x_\emptyset$, 
%\item[] $\lambda_{(24, \{5,6\})} = x_{24} + x_2 + x_4+ x_\emptyset$, 
%\item[] $\lambda_{(4, \{56\})} = x_\emptyset$, 
%\item[] $\lambda_{(4, \{5,6\})} = x_4+ x_\emptyset$,
%\end{enumerate}
%
 \begin{align*}
\lambda_{(1234,\{56\})} &= x_{14} + x_{24} + x_{34}+ x_\emptyset,  \\
\lambda_{(1234, \{5,6\})} &= x_{14} + x_{24} + x_{34} + x_1 + x_2 + x_3 + x_4+ x_\emptyset, \\ 
\lambda_{(24, \{56\})} &= x_{24}+ x_\emptyset, \\
  \lambda_{(24, \{5,6\})} &= x_{24} + x_2 + x_4+ x_\emptyset, \\
\lambda_{(4, \{56\})} &= x_\emptyset, \\
  \lambda_{(4, \{5,6\})} &= x_4+ x_\emptyset,
\end{align*} 
all with multiplicity 1. 
\begin{figure}[h]
\[\scalebox{0.5}{$\xymatrix{ & 123\ar@{-}[dr]\ar@{-}[dl] & & &  \\ 12 \ar@{-}[dd] & & 23 \ar@{-}[ddll]\ar@{-}[dd] & &  4\ar@{-}[dd] \\  & & & \hspace{3cm} &  \\ 2 \ar@{-}[dr] & & 3\ar@{-}[dl] & &  \emptyset \\ & \emptyset & & &  }$}\]
\caption{The lattices of upsets  for the posets $P_5$ and $P_6$ in Figure~\ref{treeex1}.}
\label{treeseigenvalues}
\end{figure}
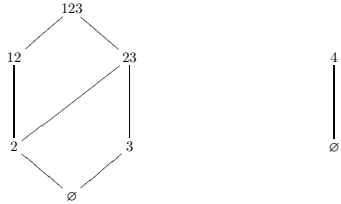

\end{exa}

\begin{proofof}{Theorem \ref{theorem: main1}} For each vertex $v$ of depth $d-1$, let $\mathcal{M}_{v}$ be the monoid generated by the transformations on $\mathcal{P_{v}}$ induced by $\{\widehat{\partial}_j :\ j \in C_{v}\}$.   Ayyer et al.~\cite{ayyer2014combinatorial} proved that when the poset on $C_{v}$ is a rooted forest, $\mathcal{M}_{v}$ is $\mathcal{R}$-trivial with an associated semilattice $L^{\mathcal{M}_{v}}$. The monoid generated by the pop shuffles on the set \[ \bigoplus_{\substack{v \in I \\ d(v) \neq d-1}} \mathcal{L}(C_{v})\] is precisely $\partition^{\text{ord}}(T\setminus L)$, which is a left-regular band and therefore $\mathcal{R}$-trivial with an associated semilattice  $\partition(T\setminus L)$~\cite{bjorner2009note}.  The support map is defined by \begin{align*}\supp \colon \mathrm{Part}^{\text{ord}}(T\setminus L) &\rightarrow \partition(T \setminus L) \\ 
\underline{\alpha} &\mapsto \alpha.\end{align*}

Thus, if $T$ has $k$ vertices of depth $d-1$, $v_{1}, \ldots, v_{k}$, the monoid $\mathcal{M}^{T}$ generated by the transformations of $\mathcal{L}(T)$ induced by the moves $\widehat{\partial}_{E}$ is a submonoid of \[\mathcal{M} = \mathcal{M}_{v_1} \times \cdots \times  \mathcal{M}_{v_k} \times \mathrm{Part}^{\text{ord}}(T\setminus L).\]  We can think of acting with the larger monoid $\mathcal{M}$ with the probabilities set to be

\begin{equation} \label{redifineprob} \mathrm{Prob}(y, \underline{\alpha}) = \begin{cases} \displaystyle \sum_{ y= \widehat{\partial}_E, \underline{\alpha} = \underline{\alpha}^E} x_E  & \text{ if there is some $E \in \mathcal{A}(L)$ such that } y= \widehat{\partial}_E \text{ and } \underline{\alpha} = \underline{\alpha}^E \\ 0 & \text{ otherwise. } \end{cases} \end{equation} Then $\mathcal{M}^{T}$ is the submonoid of $\mathcal{M}$ generated by the support of $\mathrm{Prob}$. It is not difficult to see that $\mathcal{M}^{T}$ contains an element that acts as a constant map. Since the action of $\mathcal{M}$ on $\mathcal{L}(T)$ is faithful, the minimal ideal of $\mathcal{M}^{T}$ is canonically in bijection with $\mathcal{L}(T)$ and, moreover, that bijection is an isomorphism of the action of $\mathcal{M}^{T}$ on the left of the minimal ideal with the action of $\mathcal{M}^{T}$ on $\mathcal{L}(T)$ (see for example Remark 2.8 in~\cite{ayyer2015markov}). 

As a product of $\mathcal{R}$-trivial monoids, $\mathcal{M}$ is also $\mathcal{R}$-trivial with an associated semilattice \[L^{\mathcal{M}} =  L^{\mathcal{M}_{v_1}} \times \cdots \times L^{\mathcal{M}_{v_k}} \times \partition(T \setminus L).\]

The semilattice $L^{\mathcal{M}_{v}}$ was described in~\cite{ayyer2014combinatorial} for the case when the poset on $C_{v}$ is a rooted forest as follows. For $x \in \mathcal{M}_{v}$, the image of $x$ is $\text{im}(x) = \{x \pi : \pi \in \mathcal{P_{v}}\}$. Let $\text{rfactor}(x)$ be the largest common right factor of all $\pi \in \text{im}(x)$. In other words, for $\pi \in \text{im}(x)$, $\pi = \pi'\text{rfactor}(x)$ and there is no bigger $\text{rfactor}(x)$ such that this is true. Let $\text{Rfactor}(x) = \{j : \ j \in \text{rfactor}(x)\}$. The support map is defined by \begin{align*} \supp : \mathcal{M}_{v} & \rightarrow L^{\mathcal{M}_{v}} \\ x &\mapsto \text{Rfactor}(x^{\omega}), \end{align*} where $x^{\omega}$ is such that $x^{\omega}x = x^{\omega}$ is idempotent. The $\mathcal{R}$-triviality of $\mathcal{M}_{v}$ guarantees that $x^{\omega}$ exists and is idempotent. 

The support map $\supp: \mathcal{M} \rightarrow L^{\mathcal{M}}$ is now taken component-wise. The sets $\text{Rfactor}(x)$ are upsets in the associated posets. So, we need to show that when $S$ is an upset which is not of the form $\text{Rfactor}(x)$ and thus not in $L^{\mathcal{M}}$, the multiplicity $m_{(S, \alpha)} = d_{S_1} \cdots d_{S_k} m_\alpha$ is $0$. Otherwise, we need to check that both the formulas for the eigenvalues and their multiplicities in Theorem~\ref{theorem: main1} and Theorem~\ref{theorem: rmonoid} match.

First, let $(S, \alpha)$ be such that $S$ is an upset of $P$ that is not $\text{Rfactor}(x)$ for any $x \in \mathcal{M}_{v_1} \times \cdots \times  \mathcal{M}_{v_k}$. Then there exists a component of $S$, say $S_j$, which is not of the form $\text{Rfactor}(x)$ for any $x$ in $\mathcal{M}_{v_{j}}$. By the proof of Theorem~5.2 in~\cite{ayyer2014combinatorial},  $d_{S_j} = 0$. Thus, $m_{(S, \alpha)} = 0$.

Now let $(S, \alpha) = (S_{1}, \ldots, S_{k}, \alpha) \in L^{\mathcal{M}}$. For an element $(y_{1}, \ldots, y_{k}, \underline{\beta}) \in \mathcal{M}$, $\supp(y_{1}, \ldots, y_{k}, \underline{\alpha}) \leq (S_{1}, \ldots, S_{k}, \alpha)$ if and only if $\supp(y_{i}) \leq S_{i}$, $1 \leq i \leq k$, and $\supp(\underline{\beta}) \leq \alpha$. The latter is true if and only if $\alpha$ is a refinement of $\beta$. The probability $x_{(y_{1}, \ldots, y_{k}, \underline{\beta})}$ is zero unless each $y_{i}$ is of  the form $\widehat{\partial}_{j_{i}}$ for some $j_{i} \in C_{v_{i}}$. But from~\cite{ayyer2014combinatorial} we know that $\supp(\widehat{\partial}_{j_{i}}) \leq S_{i}$ if and only if $j_{i} \in S_{i}$. So, the only nonzero summands in~\eqref{rev} correspond to $\widehat{\partial}_{E}$ such that $E \subseteq S$ is an element of $\mathcal{A}(L)$ and $\alpha$ is a refinement of $\beta^{E}$. This is precisely the definition of $\alpha$-compatible.

The multiplicity in this case, by~\eqref{rmult}, is
\begin{equation}  \label{msa1} m_{(S, \alpha)} = \sum_{\substack{S' \geq S \\ \alpha' \geq \alpha}} \mu((S, \alpha), (S', \alpha')) c_{(S', \alpha')}.\end{equation}
The M\"obius function and the number of maximal elements are both multiplicative. Thus 
\begin{equation} \label{msa2} m_{(S, \alpha)}  =   \sum_{\beta \geq \alpha}\mu(\alpha, \beta) c_{\beta} \prod_{i=1}^k \left( \sum_{S'_i \geq S_i} \mu(S_i, S'_i) c_{S'_i}\right).\end{equation} In the analysis of the extended promotion Markov chain~\cite{ayyer2014combinatorial} it's proved that 
\begin{equation} \label{msa3} \sum_{T_i \geq S_i} \mu(S_i, T_i) c_{T_i} = d_{S_i}.\end{equation} On the other hand, in the analysis of the move-to-front scheme on trees~\cite{bjorner2009note}, Bj\"orner showed that by applying Theorem 1 of Brown~\cite{brown2000semigroups}, one gets  \begin{equation} \label{msa4} \sum_{\beta \geq \alpha} \mu(\alpha, \beta) c_\beta = m_\alpha.\end{equation} Substituting~\eqref{msa3} and~\eqref{msa4} into~\eqref{msa2} yields $m_{(S, \alpha)} = d_{S_1} \cdots d_{S_k} m_\alpha$.

\end{proofof}
%%%%%%%%%%%%%%%%%%%%%%%%%%%%%%%%%%%%%%%%%%%%%%%%%%
% Ladders using DO(Ab)
%%%%%%%%%%%%%%%%%%%%%%%%%%%%%%%%%%%%%%%%%%%%%%%%%%%

Our second result in this section gives a description of the eigenvalues of the transition matrix $M^{T}$ when each of the leaf posets is a disjoint union of an ordinal sum of a forest and a ladder. More precisely, let the set of vertices of depth $d-1$ (parents of leaves) in $T$ be  $\{v_{1}, \cdots, v_{k}\}$ and let each leaf poset be of the form \[P_{v_i} = F^{i}_1 \oplus L^{i}_1 + \cdots + F^{i}_{k_{i}} \oplus L^{i}_{k_{i}}\] where each $F^{i}_j$ is a rooted forest and $L^{i}_j$ is a ladder. As in the proof of Theorem~\ref{theorem: main1}, let $M_{v_{i}}$, $1 \leq i \leq k$ be the monoids generated by the transformations on $\mathcal{L}(P_{v_{i}})$ induced by the extended promotion operators. The eigenvalues in the following theorem are described in terms of the algebraic structure of the product $\mathcal{M}_{v_1} \times \cdots \times \mathcal{M}_{v_k}$.

\begin{thm}\label{theorem: main2}
Let $T$ be as described in the previous paragraph and let $M^{T}$ be the transition matrix of the random walk on $\mathcal{L}(T)$ \[M(\pi, \pi') = \sum_{E : \widehat{\partial}_E \pi = \pi'} x_E\] for $\pi, \pi' \in \mathcal{L}(T)$. 
Choose a maximal subgroup $H_J$ of each $\mathcal{J}$-class $J$ of the  monoid $\mathcal{M}_{v_1} \times \cdots \times \mathcal{M}_{v_k}$ with identity $e_J$, and let $\alpha \in \partition(T \setminus L )$. The eigenvalues of $M^{T}$ are given by 
\begin{equation} \label{evdoabours} \lambda_{(J, \chi, \alpha)} = \sum_{\substack{E \in \mathcal{A}(L) \\ \widehat{\partial}_E \geq_{\mathcal{J}} J \\ E \text{ is } \text{$\alpha$-compatible}}} x_E \chi(e_J \widehat{\partial}_E e_J),\end{equation}  for each irreducible character $\chi$ of $H_J$. The multiplicity of $\lambda_{(J, \chi, \alpha)}$ may be zero; however, there are no other eigenvalues. 

\end{thm}

\begin{exa}\label{ex: tree with poset structure2}

Consider the tree $T$ with the leaf posets given as in Figure~\ref{treewithposet3}. $T$ has two vertices at depth $d-1$: $5$ and $6$. 
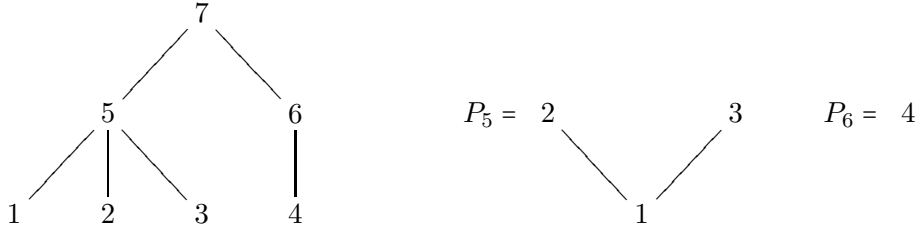
\begin{figure}[h]
\begin{center}
\scalebox{1}{$$\xymatrix{  & & 7\ar@{-}[dr]\ar@{-}[dl]&  & & & \\ & 5\ar@{-}[d]\ar@{-}[dr]\ar@{-}[dl]& & 6\ar@{-}[d]  & & P_5 = \hspace{-.3in}& 2\ar@{-}[dr] & & 3\ar@{-}[dl] & P_6 = \hspace{-.3in}& 4 \\ 1& 2& 3& 4 & & & & 1 & & & & }  $$} 
\end{center}
\caption{A rooted tree with a ladder as the leaf poset.}
\label{treewithposet3}
\end{figure}
The monoid $\mathcal{M}_{5} \times \mathcal{M}_{6}$ has two $\mathcal{J}$-classes: $J = \{(\widehat{\partial}_1, \widehat{\partial}_4), (\widehat{\partial}_\emptyset, \widehat{\partial}_4)\}$ and  $J'=\{(\widehat{\partial}_2, \widehat{\partial}_4), (\widehat{\partial}_3, \widehat{\partial}_4)\}$. Note that on the second component $\widehat{\partial}_\emptyset = \widehat{\partial}_4$. The corresponding maximal subgroups are $H_{J} = J$ and $H_{J'}= \{(\widehat{\partial}_2, \widehat{\partial}_4)\}$ with identity elements $e_J = (\widehat{\partial}_\emptyset, \widehat{\partial}_4)$ and $e_{J'} = (\widehat{\partial}_2, \widehat{\partial}_4)$, respectively.
The character tables for $H_J$ and $H_{J'}$ are
\begin{center}
\begin{tabular}{c| c c }
&  $(\widehat{\partial}_\emptyset, \widehat{\partial}_4) $& $(\widehat{\partial}_1, \widehat{\partial}_4)$\\
\hline
$\chi^{(1)}$ & $1$ & $1$ \\
 $\chi^{(2)}$& $1$  & $-1$\\ 
\end{tabular}
\qquad \qquad \qquad
\begin{tabular}{c| c }
&   $(\widehat{\partial}_2, \widehat{\partial}_4)$\\
\hline
$\chi^{(1)}$ & $1$
\end{tabular}.
\end{center}

The inequality $\widehat{\partial}_E \geq_{\mathcal{J}} J$ is satisfied for $E \in \{\emptyset, \{1\}, \{4\}, \{1,4\}\}$ while $\widehat{\partial}_E \geq_{\mathcal{J}} J'$ for all $E \in \mathcal{A}(L)$. Based on the requirement to be $\alpha$-compatible, the possible eigenvalues are 
\begin{align*}
\lambda_{(J, \chi^{(1)}, \{56\})} &= x_{14} +x_\emptyset \\
\lambda_{(J, \chi^{(1)}, \{5,6\})} &=x_{14} + x_1 + x_4 + x_\emptyset  \\
\lambda_{(J, \chi^{(2)}, \{56\})} &=  -x_{14} + x_\emptyset \\
\lambda_{(J, \chi^{(2)}, \{5,6\})} &=  - x_1 - x_{14} + x_4  + x_\emptyset \\ 
\lambda_{(J', \chi^{(1)}, \{56\})} &= x_{14} + x_{24} + x_{34} + x_\emptyset \\ 
\lambda_{(J', \chi^{(1)}, \{5,6\})} &= x_{1} + x_2 + x_3 + x_4 + x_{14} + x_{24} + x_{34} + x_\emptyset.
\end{align*}
%\begin{enumerate}
%\item[] $\lambda_{(J, \chi^{(1)}, \{56\})} = x_{14} +x_\emptyset $
%\item[] $\lambda_{(J, \chi^{(1)}, \{5,6\})} =x_{14} + x_1 + x_4 + x_\emptyset  $
%\item[] $\lambda_{(J, \chi^{(2)}, \{56\})} =  -x_{14} + x_\emptyset $
%\item[] $\lambda_{(J, \chi^{(2)}, \{5,6\})} =  - x_1 - x_{14} + x_4  + x_\emptyset$
%\item[] $\lambda_{(J', \chi^{(1)}, \{56\})} = x_{14} + x_{24} + x_{34} + x_\emptyset$
%\item[] $\lambda_{(J', \chi^{(1)}, \{5,6\})} = x_{1} + x_2 + x_3 + x_4 + x_{14} + x_{24} + x_{34} + x_\emptyset.$
%\end{enumerate}
Based on formula~\eqref{mdoab}, the multiplicities of  $\lambda_{(J, \chi^{(1)}, \{56\})}$ and $\lambda_{(J, \chi^{(1)}, \{5,6\})}$ are actually $0$ while the remaining four eigenvalues have multiplicity one.

\end{exa}

\begin{proofof}{Theorem \ref{theorem: main2}}

As in the proof of Theorem~\ref{theorem: main1}, it is equivalent to consider the Markov chain given by the actions of the elements in the monoid $\mathcal{M} = \mathcal{M}_{v_1} \times \cdots \times  \mathcal{M}_{v_k} \times \mathrm{Part}^{\text{ord}}(T\setminus L)$ with probabilities given by~\eqref{redifineprob}. Recall that $\mathcal{M}_{v_{i}}$ is the monoid generated by the transformations on $C_{v_{i}}$ induced by the extended promotion operators, or equivalently, the matrices $G_{j}$ given by substituting all indeterminants but $x_{j}$ in the transition matrix of the extended promotion Markov chain on $\mathcal{L}(C_{v_{i}})$ by 0 and setting $x_{j} =1$. By the proof of Theorem 5 of the authors in~\cite{poznanovic2017properties}, when the leaf posets are as assumed, there exists a matrix $U$ that simultaneously upper-triangularizes all $G_j$, which implies that $\mathcal{M}_{v_{i}}$ is in \textbf{DO(Ab)} by Theorem~\ref{complexrep}. Therefore, $\mathcal{M}$ is also in \textbf{DO(Ab)}.

So, one can use Theorem~\ref{theorem: doab} to find the desired eigenvalues. A $\mathcal{J}$-class of a direct product is a direct product of $\mathcal{J}$-classes. The analogous statement is true for the maximal subgroups. Two ordered partitions of $T \setminus L$, $\underline{\alpha}$ and $\underline{\beta}$, are in the same $\mathcal{J}$-class of $\mathrm{Part}^{\text{ord}}(T\setminus L)$ if their underlying set partitions $\alpha$ and $\beta$ are equal. So, the $\mathcal{J}$-classes of $\mathrm{Part}^{\text{ord}}(T\setminus L)$ are indexed by set partitions. Since $\mathrm{Part}^{\text{ord}}(T\setminus L)$ is right-regular band, the maximal subgroup of such a $\mathcal{J}$-class, $J_\alpha$ is trivial, and can be thought of as $H_{J_{\alpha}} = \{\underline{\alpha}\}$. So, the characters in~\eqref{complexrep} reduce to characters of the maximal subgroups for $\mathcal{M}_{v_1} \times \cdots \times  \mathcal{M}_{v_k}$. Now, similarly as in the proof of Theorem~\ref{theorem: main1}, $\underline{\beta}^{E} \geq J_{\alpha}$ if and only if $E$ is $\alpha$-compatible and we get~\eqref{evdoabours}.

\end{proofof}

\begin{rem} In the proofs of Theorem~\ref{theorem: main1} and~\ref{theorem: main2} we used that  the monoid $\mathcal{M}^T$ generated by the transformations on $\mathcal{L}(T)$ induced by $\widehat{\partial}_{E}$, $E \in \mathcal{A}(L)$ embeds into $\mathcal{M} = \mathcal{M}_{v_1} \times \cdots \times \mathcal{M}_{v_k} \times \mathrm{Part}^{\text{ord}}(T\setminus L)$. While the embedding is in general proper, $\mathcal{M}^T$ and $\mathcal{M}$ are equal if the leaf posets are rooted trees. To see this, consider a generator $(\widehat{\partial}_E, \underline{\alpha}^F)$ of $\mathcal{M}$, where $F$ is also an element of $\mathcal{A}(L)$ and $\underline{\alpha}^F  \in \mathrm{Part}^{\text{ord}}(T\setminus L)$  whose component $\underline{\alpha}_{v}^F$ at an inner vertex $v$ of $T \setminus L$ is the one- or two-block ordered partition of $C_{v}$ with the last block consisting of the $F$-related children of $v$. (The fact that one- and two-block partitions generate  $\mathrm{Part}^{\text{ord}}(T\setminus L)$ follows for example from the discussion in~\cite{bjorner2009note}.) Let $r_{i}$ be the root of the poset $P_{v_{i}}$ on the set of leaves $C_{v_{i}}$, $1 \leq i \leq k$. Let
\[G = \{ v \in L : v \in (E \cap C_{v_{i}}) \setminus F\} \cup \{r_{i} : E \cap C_{v_{i}} = \emptyset \text{ or } |F \cap E \cap C_{v_{i}}| = 1\}. \]
By construction, $|G \cap C_{v_{i}}| = 1$ for every $i \in \{1, \cdots, k\}$ and hence every inner node of $T$ is $G$-related and $\underline{\alpha}^{G}$ acts as an identity on the linear extensions of the inner nodes of $T$.  Let
\[H  = \bigcup_{|F\cap E \cap C_{v_{i}}| =1} (F\cap E \cap C_{v_{i}}) \cup \{ r_{i} : |F \cap C_{v_{i}}| =1 \text{ but } F \cap E \cap C_{v_{i}} = \emptyset \}.\]
By construction, $|H \cap C_{v_{i}}| = |F \cap C_{v_{i}}|$ for all $i$, hence $\underline{\alpha}^{H} = \underline{\alpha}^{F}$. Since $r_{v_{i}}$ is the root, the action of $\widehat{\partial}_{r_{i}}$ on $\mathcal{L}(C_{v_{i}})$ is trivial. Based on this, one can now readily see  that
\[(\widehat{\partial}_H, \underline{\alpha}^H) \cdot (\widehat{\partial}_G, \underline{\alpha}^G) = (\widehat{\partial}_E, \underline{\alpha}^F)\] and, consequently, $\mathcal{M} \subseteq \mathcal{M}^{T}$.
\end{rem}

%%%%%%%%%%%%%%%%%%%%%%%%%%%%%%%%%%%%%%%%%%%%%%%

\section{The eigenvalues of $M^{T_{P}}$: a combinatorial treatment}\label{alternate}

While Theorem~\ref{theorem: main1} gives a combinatorial description of the eigenvalues of the transition matrix $M^T$ when the leaf posets are forests, in the more general case when the leaf posets are of the form $F_1 \oplus L_1 + \cdots + F_k \oplus L_k$, for $F_{i}$ a forest and $L_{i}$ a ladder, the description of the eigenvalues in Theorem~\ref{theorem: main2}  is given in terms of the algebraic properties of the underlying monoid. The goal of this section is to give a combinatorial description of the eigenvalues in this more general case. More precisely, we give an algorithm how to compute the eigenvalues (including multiplicities) from which it is clear that the character values that appear in Theorem~\ref{theorem: main2} are all $\pm 1$. The main result in this section is summarized in the following theorem.

\begin{thm}\label{theorem: maindoab}
Let $T$ be as described above and suppose the leaf posets $P_{v_i}$ are all of the form  $P_{v_i} = F^{i}_1 \oplus L^{i}_1 + \cdots + F^{i}_{k_{i}} \oplus L^{i}_{k_{i}}$ where each $F^{i}_j$ is a rooted forest and $L^{i}_j$ is a ladder. The eigenvalues of the transition matrix $M^T$ are linear in the $x_E$'s. Moreover, they can be explicitly computed using the formula for the case when $P_{v_i}$ is a rooted forest (Theorem~\ref{theorem: main1}) and Theorem~\ref{generalfactor}.
\end{thm}

The proofs and definitions in this section are in the same spirit as the proofs and definitions in~\cite{poznanovic2017properties} where we analyzed the promotion Markov chain for posets of the form $F_1 \oplus L_1 + \cdots + F_k \oplus L_k$.  The idea is that if $P=F_1 \oplus L_1 + \cdots + F_k \oplus L_k$ where $F_i$ is a forest and $L_i$ is a ladder, then $P$ can be obtained by starting from a forest in which the upper parts of the tree components are chains and then breaking covering relations in the chains one by one to obtain the desired ladders. Since we will need to compare two Markov chains with the same underlying tree $T$ but different leaf posets $P$ and $P'$, in this section we will denote the corresponding  transition matrices $M^{T_{P}}$ and $M^{T_{P'}}$, respectively.
\begin{exa}
Let $P$ be the poset in Figure~\ref{startingforest}. If only the covering relation $2 \prec 3$ is excluded one obtains the poset $P' = P_{5} + P_{6}$ from Figure~\ref{treewithposet3}. By Theorem~\ref{theorem: main1}, the eigenvalues of $M^{T_P}$ are  \[\lambda_1 = x_{14} + x_{24} + x_{34}+x_\emptyset, \; \; \; \lambda_2= x_{14} + x_{24} + x_{34} + x_1 + x_2 + x_3 + x_4+x_\emptyset.\] As we have seen in Example~\ref{ex: tree with poset structure2}, the eigenvalues of $M^{T_P'}$ are \[\lambda_1=x_{14} + x_{24} + x_{34}+x_\emptyset, \; \; \; \lambda_1'=-x_{14}+x_\emptyset, \; \; \;\lambda_2= x_{14} + x_{24} + x_{34} + x_1 + x_2 + x_3 + x_4+x_\emptyset, \; \; \;  \lambda_2'=x_4+x_\emptyset - x_1 - x_{14}. \]  As we will see below in Theorem~\ref{generalfactor}, the eigenvalues $\lambda_1$ and $\lambda_1'$ of $M^{T_{P'}}$ correspond to the eigenvalue $\lambda_1$ of $M^{T_P}$  and the eigenvalues $\lambda_2$ and $\lambda_2'$ of $M^{T_{P'}}$ correspond to the eigenvalue $\lambda_2$ of $M^{T_P}$. 
\begin{figure}[h]
\[\scalebox{0.7}{\xymatrix{ 3\ar@{-}[d]  & \hspace{-.7in}\bullet & 4 & \hspace{-.7in} \bullet \\ 2\ar@{-}[d]  & \hspace{-.7in} \bullet & & \\ 1 &  \hspace{-.7in}\bullet & & }}\]
\caption{A forest $P$ related to the leaf posets in Figure~\ref{treewithposet3}.} 
\label{startingforest}
\end{figure}
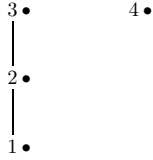
\end{exa}

To formalize the correspondence illustrated in the previous example we introduce some notation. Let $R_{P}$ be the set of all pairs $(a,b)$ for which $P$ can be written in the form  \[P = Q' \oplus a\oplus b \oplus Q''+P_{2}.\] For example, for $P$ in Figure~\ref{startingforest}, $R_P = \{(1,2), (2,3)\}$. Denote by $P' = P \setminus \{(a, b)\}$ the poset obtained by excluding the covering relation $a \prec b$, i.e.,  $P' = Q' \oplus ( a + b ) \oplus Q'' +P_{2}$. For a linear combination $x^{\mathfrak{s}}$ of $x_{E}$'s, we will write $x_{E} \in x^{\mathfrak{s}}$ to denote that $x_{E}$ appears in $x^{\mathfrak{s}}$ with a nonzero coefficient. 

We will say that $M^{T_P}$ has the \emph{upset property} if its characteristic polynomial factors into linear terms, and for each eigenvalue $x^{\mathfrak{s}} = \sum c_{E}^\mathfrak{s}x_{E}$ of $M^{T_{P}}$ and a pair of vertices $(a,b) \in R_{P}$ with common parent $v$, one of the following two conditions holds:
\begin{enumerate}[(A)]
\item  $c^{\mathfrak{s}}_{E \cup \{a\}} = c^{\mathfrak{s}}_{E \cup \{b\}}$ for all $E$ such that $E \cap C_v = \emptyset$, or
\item $c^{\mathfrak{s}}_{E \cup \{k\}} = 0$ for all $k \preceq a$ for all $E$ such that $E \cap C_v = \emptyset$. 
\end{enumerate}

%Let $v$ be such that $a,b \in C_v$. 

Let $T_P$ and $T_{P'}$ be two trees that have the same underlying structure, but whose leaf posets satisfy $P'= P \setminus \{(a,b)\}$ for some $(a,b) \in R_{P}$. 

 \begin{thm}\label{generalfactor}
Suppose the leaf poset is of the form  $P = Q' \oplus a\oplus b \oplus Q''+P_{2}$ and $a,b \in C_{v}$. Let $P' = P\setminus \{(a, b)\}$. Suppose there is a complex invertible matrix $S$ (independent of the $x_{E}$'s) such that $SM^{T_{P}}S^{-1}$ is upper triangular and $M^{T_P}$ has the upset property. Then $M^{T_{P'}}$ is also uppertriangularizable by a complex matrix and has the upset property. Moreover, for each eigenvalue $x^{\mathfrak{s}} = \sum c_{E}^\mathfrak{s}x_{E}$ of $M^{T_P}$, $M^{T_{P'}}$ has two eigenvalues given by

\[ \begin{cases} x^{\mathfrak{s}}, \; \; \displaystyle \sum_{\substack{k \in C_v \\ k \not\preceq a, b}} \sum_{E} c_{E\cup \{k\}}^{\mathfrak{s}}x_{E\cup \{k\}} + \sum_{E} c_E^{\mathfrak{s}}x_E - \sum_{\substack{k \in C_v \\ k \prec a}} \sum_{E} c_{E\cup \{k\}}^{\mathfrak{s}}x_{E\cup \{k\}}  & \; \;  \text{ if } x^{\mathfrak{s}}, a,b \text{ satisfy property (A)} \\%c^{\mathfrak{s}}_{E \cup \{a\}} = c^{\mathfrak{s}}_{E \cup \{b\}} \\
x^{\mathfrak{s}}, \; \; \displaystyle \sum_{\substack{k \in C_v \\ k \not\preceq a, b}} \sum_{E} c_{E\cup \{k\}}^{\mathfrak{s}}x_{E\cup \{k\}} + \sum_{E} c_E^{\mathfrak{s}}x_E  + \sum_{E} c^{\mathfrak{s}}_{E \cup\{b\}} x_{E \cup \{a\}}  & \; \;  \text{ if } x^{\mathfrak{s}}, a,b \text{ satisfy property (B),}%c^{\mathfrak{s}}_{E \cup \{k\}} = 0 \text{ for all } k \preceq a. 
\end{cases}
\]
where the sets $E$ in the sums above vary over the set $\{ E : E \in \mathcal{A}(L), E \cap C_{v} = \emptyset\}$.

 \end{thm}

Let  $G_{E}$ be the matrix obtained from $M^{T_{P}}$ by setting $x_{E} = 1$ and all other probability parameters 0, i.e., $M^{T_{P}} = \sum x_{E} G_{E}$. When the leaf poset $P$ is a forest,  the transition matrix $M^{T_P}$ satisfies the assumptions of Theorem~\ref{generalfactor}, because the monoid generated by the matrices $G_E$ is $\mathcal{R}$-trivial and the eigenvalues of the transition matrix are supported on the upsets of the tree (Theorem~\ref{theorem: main1}). Thus, by starting with an appropriate forest and repeatedly applying Theorem~\ref{generalfactor}, one can find the eigenvalues in the case when the leaf poset is of the form $F_1 \oplus L_1 + \cdots + F_k \oplus L_k$.
% Theorem~\ref{theorem: maindoab} follows directly from Theorem~\ref{generalfactor}.

The rest of this section is devoted to the proof of Theorem~\ref{generalfactor} which is based on several lemmas that we prove first. For $\pi \in \mathcal{L}(T_{P})$, let $\hat{\pi} \in \mathcal{L}(T_{P'})$ be the linear extension of $T$ obtained by interchanging $a$ and $b$. Then
\[ \mathcal{L}(T_{P'}) =  \{ \pi,  \hat{\pi} : \pi \in \mathcal{L}(T_{P}) \}.\] the matrices $M^{T_P}$ and $M^{T_{P'}}$ are closely related as described in the following lemma, which is analogous to Lemma~14 in~\cite{poznanovic2017properties}.

%Assume $P$ and $P'$ are as in the stateNotice that For the rooted trees $T_P$ and $T_{P'}$, 

%For the posets $P$ and $P'$ described at the beginning of this section, and 

\begin{lem}\label{form2}
Let $P = Q' \oplus a\oplus b \oplus Q''+P_{2}$ and let $P' = P\setminus \{(a, b)\}$ be two leaf posets for the tree $T$ where $a,b \in C_v$. Let $G_{T_P}$ and $G_{T_{P'}}$ be the labeled directed graphs that depict the moves in the Markov chains given with the transition matrices $M^{T_P}$ and $M^{T_{P'}}$, respectively. Then 

\begin{enumerate}[(1)]
\item \label{one} If $E \cap C_v = \{k\}$, $k \prec a$ and $\pi \overset{x_{E}}{\rightarrow} \widetilde{\pi}$ in $G_{T_P}$, then $\pi\overset{x_{E}}{\rightarrow} \widehat{\widetilde{\pi}}$ and $\widehat{\pi}\overset{x_{E}}{\rightarrow}\widetilde{\pi}$ in $G_{T_{P'}}$.
\item \label{two} If $E \cap C_v = \emptyset$ or $E \cap C_v = \{k\}$, $k \npreceq a$, $k \npreceq b$ and $\pi \overset{x_E}{\rightarrow} \widetilde{\pi}$ in $G_{T_P}$, then $\pi\overset{x_E}{\rightarrow}\widetilde{\pi}$ and $\widehat{\pi}\overset{x_E}{\rightarrow} \widehat{\widetilde{\pi}}$ in $G_{T_{P'}}$.
\item \label{three} If $E \cap C_v = \{a\}$ and $\pi \overset{x_E}{\rightarrow} \widetilde{\pi}$ in $G_{T_P}$, then $\pi\overset{x_{E}}{\rightarrow}\widehat{\widetilde{\pi}}$ and $\widehat{\pi}\overset{x_{(E \setminus \{a\}) \cup \{b\}}}{\longrightarrow}\widetilde{\pi}$ in $G_{T_{P'}}$.
\item \label{four} If $E \cap C_v = \{b\}$ and $\pi \overset{x_E}{\rightarrow} \widetilde{\pi}$ in $G_{T_P}$, then $\pi\overset{x_E}{\rightarrow}\widetilde{\pi}$ and $\widehat{\pi}\overset{x_{(E \setminus \{b\}) \cup \{a\}}}{\longrightarrow}\widehat{\widetilde{\pi}}$ in $G_{T_{P'}}$.
\end{enumerate}
\end{lem}

\begin{proof}
Notice that if $a,b \in C_v$, then $\pi_w = \hat{\pi}_w$ for all vertices $w\neq v$. Thus if $\pi \overset{x_{E}}{\rightarrow} \widetilde{\pi}$ in $G_{T_P}$, then $\pi_w \overset{x_{E}}{\rightarrow} \widetilde{\pi}_w$ in $G_{T_{P'}}$. So the only component we need to consider is $\pi_v$. Since $E \in \mathcal{A}(L)$ contains at most one element from $C_v$, the rules follow  from Lemma~14 in~\cite{poznanovic2017properties}.

\end{proof}

%Let $T_P$ be a rooted tree with a size-$n$ leaf poset $P= Q' \oplus a\oplus b \oplus Q''+P_{2}$. 
For the size-$m$ transition matrix $M^{T_P}$, we denote by $\partial_{a,b} M^{T_P}$ the $2m \times 2m$ matrix obtained by replacing each entry of $M^{T_P}$ by a $2 \times 2$ block using the linear extension of the map:
\begin{align*}
    \begin{blockarray}{cc}
& \widetilde{\pi} \\
\begin{block}{c(c)}
\pi & x_{E} \\
\end{block}
\end{blockarray}
\longmapsto
\begin{blockarray}{ccc}
& \widetilde{\pi} & \widehat{\widetilde{\pi}} \\
\begin{block}{c(cc)}
\pi &  & x_{E}\\
\widehat{\pi} & x_{E} &  \\
\end{block}
\end{blockarray}
\qquad &\text{ if } E \cap C_v = \{k\}, k \prec a  \\
       \begin{blockarray}{cc}
& \widetilde{\pi} \\
\begin{block}{c(c)}
\pi & x_{E} \\
\end{block}
\end{blockarray}
\longmapsto
\begin{blockarray}{ccc}
& \widetilde{\pi} & \widehat{\widetilde{\pi}} \\
\begin{block}{c(cc)}
\pi &   x_{E} & \\
\widehat{\pi} & &  x_{E}   \\
\end{block}
\end{blockarray}
\qquad &\text{ if } E \cap C_v = \emptyset \; \text{ or } \; E \cap C_v = \{k\}, k \not\preceq a, k \not\preceq b  \\
       \begin{blockarray}{cc}
& \widetilde{\pi} \\
\begin{block}{c(c)}
\pi & x_{E} \\
\end{block}
\end{blockarray}
\longmapsto
\begin{blockarray}{ccc}
& \widetilde{\pi} & \widehat{\widetilde{\pi}} \\
\begin{block}{c(cc)}
\pi &  & x_{E}\\
\widehat{\pi} & x_{(E \setminus \{a\}) \cup \{b\}}&  \\
\end{block}
\end{blockarray}
\qquad &\text{ if } E \cap C_v = \{a\} \\
     \begin{blockarray}{cc}
& \widetilde{\pi} \\
\begin{block}{c(c)}
\pi & x_{E} \\
\end{block}
\end{blockarray}
\longmapsto
\begin{blockarray}{ccc}
& \widetilde{\pi} & \widehat{\widetilde{\pi}} \\
\begin{block}{c(cc)}
\pi &   x_{E} & \\
\widehat{\pi} & & x_{(E \setminus \{b\}) \cup \{a\}}   \\
\end{block}
\end{blockarray}
\qquad &\text{ if } E \cap C_v = \{b\}
\end{align*}

%\begin{enumerate}
%\item[(1)] $\begin{blockarray}{cc}
%& \widetilde{\pi} \\
%\begin{block}{c(c)}
%\pi & x_{E} \\
%\end{block}
%\end{blockarray}
%\longmapsto
%\begin{blockarray}{ccc}
%& \widetilde{\pi} & \widehat{\widetilde{\pi}} \\
%\begin{block}{c(cc)}
%\pi &  & x_{E}\\
%\widehat{\pi} & x_{E} &  \\
%\end{block}
%\end{blockarray}
%\qquad \text{ if } E \cap C_v = \{k\}, k \prec a $
%
%
%
%\item[(2)]$\begin{blockarray}{cc}
%& \widetilde{\pi} \\
%\begin{block}{c(c)}
%\pi & x_{E} \\
%\end{block}
%\end{blockarray}
%\longmapsto
%\begin{blockarray}{ccc}
%& \widetilde{\pi} & \widehat{\widetilde{\pi}} \\
%\begin{block}{c(cc)}
%\pi &   x_{E} & \\
%\widehat{\pi} & &  x_{E}   \\
%\end{block}
%\end{blockarray}
%\qquad \text{ if } E \cap C_v = \{k\}, k \not\preceq a, b \qquad \text{ or }  \qquad E \cap C_v = \emptyset$
%
%
%\item[(3)]
%$\begin{blockarray}{cc}
%& \widetilde{\pi} \\
%\begin{block}{c(c)}
%\pi & x_{E} \\
%\end{block}
%\end{blockarray}
%\longmapsto
%\begin{blockarray}{ccc}
%& \widetilde{\pi} & \widehat{\widetilde{\pi}} \\
%\begin{block}{c(cc)}
%\pi &  & x_{E}\\
%\widehat{\pi} & x_{(E \setminus \{a\}) \cup \{b\}}&  \\
%\end{block}
%\end{blockarray}
%\qquad \text{ if } E \cap C_v = \{a\}$
%
%\item[(4)]$
%\begin{blockarray}{cc}
%& \widetilde{\pi} \\
%\begin{block}{c(c)}
%\pi & x_{E} \\
%\end{block}
%\end{blockarray}
%\longmapsto
%\begin{blockarray}{ccc}
%& \widetilde{\pi} & \widehat{\widetilde{\pi}} \\
%\begin{block}{c(cc)}
%\pi &   x_{E} & \\
%\widehat{\pi} & & x_{(E \setminus \{b\}) \cup \{a\}}   \\
%\end{block}
%\end{blockarray}
%\qquad \text{ if } E \cap C_v = \{b\}$
%
%\end{enumerate}

\begin{cor} \label{formcor2} Let $T_P$ be a rooted tree with leaf poset $P = Q' \oplus a\oplus b \oplus Q''+P_{2}$ and $T_{P'}
$ be the tree with leaf poset $P' = P\setminus \{(a, b)\}$. Then $M^{T_{P'}} = \partial_{a,b} M^{T_P}$.
\end{cor}

For a complex matrix $S$, define $\partial S = S \otimes I_{2}$. So, if  $\mathcal{E}$ is an elementary matrix of size $m$ corresponding to a row operation $R$, then $\partial \mathcal{E}$ describes performing a corresponding operation to 2 rows on a matrix of size $2m$. Note that the remaining proofs in this section follow exactly from the analogous proofs in~\cite{poznanovic2017properties}.

\begin{lem}\label{diagramlemma2}
Let $S$ be a matrix with complex entries. and $M$ a matrix whose entries are homogeneous degree-1 polynomials in the $x_E$'s where $E \in \mathcal{A}(L)$.  Then \[(\partial S) (\partial_{a,b} M) =\partial_{a,b} (SM) \; \; \text{     and     } \; \; (\partial_{a,b} M) (\partial S) =\partial_{a,b} (MS) .\]
\end{lem}

\begin{proof}
For easier notation, let \begin{align*} \mathcal{A}(L)^{I} &= \{E \in \mathcal{A}(L) : E \cap C_v = \{k\}, k \prec a\}, \\ \mathcal{A}(L)^{II} &= \{E \in \mathcal{A}(L) : E \cap C_v = \emptyset \; \text{ or } \; E \cap C_v = \{k\}, k \not\preceq a, k \not\preceq b\}, \\ \mathcal{A}(L)^{III} &= \{E \in \mathcal{A}(L) : E \cap C_v = \{a\}\}, \\ \mathcal{A}(L)^{IV} &= \{E \in \mathcal{A}(L) : E \cap C_v = \{b\}\}. \end{align*} 

We can notice that the definition of $\partial_{a,b} M$ can be restated as 

\begin{align*} \partial_{a,b} M & = M\bigg \rvert_{\substack{x_{E}=0 \\ E \notin \mathcal{A}(L)^{I}} } \otimes \scalebox{1}{$\begin{pmatrix} 0 & 1 \\ 1 & 0\end{pmatrix}$} + M\bigg \rvert_{\substack{x_{E}=0 \\ E \notin \mathcal{A}(L)^{II}} } \otimes I_{2} +  \displaystyle \sum_{E \in \mathcal{A}(L)^{III}} \left[ \frac{1}{ x_{E}} M\bigg \rvert_{\substack{x_{E'}=0 \\ E' \neq E} } \otimes\scalebox{1}{$ \begin{pmatrix} 0 & x_E \\ x_{(E \setminus \{a\}) \cup \{b\}} & 0 \end{pmatrix}$} \right] \\ & \qquad \qquad  + \displaystyle \sum_{E \in \mathcal{A}(L)^{IV}} \left[ \frac{1}{ x_E} M \bigg \rvert_{\substack{x_{E'}=0 \\ E' \neq E} } \otimes   \scalebox{1}{$\begin{pmatrix} x_E & 0 \\ 0 &  x_{(E\setminus \{b\}) \cup \{a\}} \end{pmatrix}$}\right].\end{align*}

The claim follows since for a complex matrix $S$ independent of the $x_{E}$'s, \[SM\bigg \rvert_{\substack{x_{E}=0 \\ E \notin \mathcal{A}(L)^{I}}}  = (SM) \bigg \rvert_{\substack{x_{E}=0 \\ E \notin \mathcal{A}(L)^{I}}},\] etc. 
\end{proof}

%\begin{align*} \partial_{a,b} M & = M\bigg \rvert_{\substack{x_{E}=0 \\ k \in E, k \not \prec a} } \otimes \scalebox{1}{$\begin{pmatrix} 0 & 1 \\ 1 & 0\end{pmatrix}$} + M\bigg \rvert_{\substack{x_{E}=0 \\ k \in E, k \preceq b} } \otimes I_{2} +  \displaystyle \sum_{E : a \in E} \left[ \frac{1}{ x_{E}} M\bigg \rvert_{\substack{x_{E}=0 \\ k \in E, k \neq a} } \otimes\scalebox{1}{$ \begin{pmatrix} 0 & x_E \\ x_{(E \setminus \{a\}) \cup \{b\}} & 0 \end{pmatrix}$} \right] \\ & \qquad \qquad  + \displaystyle \sum_{E : b \in E} \left[ \frac{1}{ x_E} M \bigg \rvert_{\substack{x_{E}=0 \\ k \in E, k \neq b} } \otimes   \scalebox{1}{$\begin{pmatrix} x_E & 0 \\ 0 &  x_{(E\setminus \{b\}) \cup \{a\}} \end{pmatrix}$}\right].\end{align*} 

\begin{lem} \label{jordan2} Let $M$ be a matrix whose entries are homogeneous degree-1 polynomials in the $x_E$'s and let $S$ be a complex matrix such that $U = SMS^{-1}$ is upper triangular. Then the eigenvalues of $\partial_{a,b}M$ are the same as the eigenvalues of $\partial_{a,b}U$.
\end{lem}

\begin{proof} This is a direct consequence of Lemma~\ref{diagramlemma2}.
\end{proof}

\begin{proofof}{Theorem \ref{generalfactor}} 
By Corollary~\ref{formcor2}, $M^{T_{P'}} = \partial_{a,b}M^{T_P}$. By assumption, the matrix $U = SM^{T_{P}}S^{-1}$ is an upper triangular matrix whose diagonal entries are the eigenvalues $x^{\mathfrak{s}}$ of $M^{T_P}$. By Lemma~\ref{jordan2} the eigenvalues of $M^{T_{P'}}$ are the same as the eigenvalues of $\partial_{a,b}U$ which is block upper triangular with $2 \times 2$ blocks $\partial_{a,b}x^{\mathfrak{s}}$ on the main diagonal, which, using the notation from above, are

\[ \partial_{a,b}x^{\mathfrak{s}} = \scalebox{.7}{$\begin{pmatrix*}[c]  \displaystyle \sum_{E \in \mathcal{A}(L)^{II}} c_E^{\mathfrak{s}}x_E +  \sum_{E \in \mathcal{A}(L)^{IV}} c_E^{\mathfrak{s}}x_E & \displaystyle \sum_{E \in \mathcal{A}(L)^{I}} c_E^{\mathfrak{s}}x_E +  \sum_{E \in \mathcal{A}(L)^{III}} c_E^{\mathfrak{s}}x_E\\  \displaystyle \sum_{E \in \mathcal{A}(L)^{I}} c_E^{\mathfrak{s}}x_E + \sum_{E \in \mathcal{A}(L)^{III}} c_{E}^{\mathfrak{s}}x_{(E \setminus \{a\}) \cup \{b\}}  & \displaystyle \sum_{E \in \mathcal{A}(L)^{II}} c_E^{\mathfrak{s}}x_E + \sum_{E \in \mathcal{A}(L)^{IV}} c_{E}^{\mathfrak{s}}x_{(E \setminus \{b\}) \cup \{a\}} \end{pmatrix*}$}.\]

%\[ \partial_{a,b}x^{\mathfrak{s}} = \scalebox{.7}{$\begin{pmatrix*}[c]  \displaystyle \sum_{E : b \in E} c_E^{\mathfrak{s}}x_E +  \sum_{\substack{k \in C_v \\ k \nprec a,b}} \sum_{E : k \in E} c_{E}^{\mathfrak{s}}x_{E} + \sum_{E \cap C_v = \emptyset} c_E^{\mathfrak{s}}x_E  & \displaystyle \sum_{E : a \in E} c_{E}^{\mathfrak{s}}x_{E} + \sum_{\substack{k \in C_v \\ k \prec a}}\sum_{E : k \in E} c_{E}^{\mathfrak{s}}x_{E}  \\  \displaystyle \sum_{E : a \in E} c_{E}^{\mathfrak{s}}x_{(E \setminus \{a\}) \cup \{b\}} + \sum_{\substack{k \in C_v \\ k \prec a}}\sum_{E : k \in E} c_{E}^{\mathfrak{s}}x_{E}   &  \displaystyle  \sum_{E : b \in E} c_{E}^{\mathfrak{s}}x_{(E \setminus \{b\}) \cup \{a\}} + \sum_{\substack{k \in C_v \\ k \nprec a,b}} \sum_{E : k \in E} c_{E}^{\mathfrak{s}}x_{E} + \sum_{E \cap C_v = \emptyset} c_E^{\mathfrak{s}}x_E\end{pmatrix*}$}.\]
%Since by assumption, $M^{T_P}$ has the upset property, there are only two cases: 
%
%For all $E$ such that $E \cap C_v = \emptyset$, where $a, b \in C_v$
%\begin{enumerate}
%\item[(1)] $c_{E \cup \{a\}}^{\mathfrak{s}} = c_{E \cup \{b\}}^{\mathfrak{s}}$ 
%\item[] or 
%\item[(2)] $c_{E \cup \{k\}}^{\mathfrak{s}} = 0$ for all $k \preceq a$. 
%\end{enumerate}

If $x^{\mathfrak{s}}, a, b$ satisfy property (A) then \begin{align*} \sum_{E \in \mathcal{A}(L)^{III}} c_E^{\mathfrak{s}}x_E & = \sum_{E \in \mathcal{A}(L)^{IV}} c_{E}^{\mathfrak{s}}x_{(E \setminus \{b\}) \cup \{a\}}  \\  
\sum_{E \in \mathcal{A}(L)^{IV}} c_E^{\mathfrak{s}}x_E &= \sum_{E \in \mathcal{A}(L)^{III}} c_{E}^{\mathfrak{s}}x_{(E \setminus \{a\}) \cup \{b\}}    \end{align*} and so,
\[ \scalebox{1}{$ \begin{pmatrix*}[r] 1 & 0 \\ -1 & 1   \end{pmatrix*} $}\partial_{a,b}x^{\mathfrak{s}} \scalebox{1}{$ \begin{pmatrix*}[r] 1 & 0 \\ -1 & 1   \end{pmatrix*}^{-1} $}=  \scalebox{.8}{$\begin{pmatrix*}[c] x^\mathfrak{s}  & \displaystyle \sum_{E \in \mathcal{A}(L)^{I}} c_E^{\mathfrak{s}}x_E +  \sum_{E \in \mathcal{A}(L)^{III}} c_E^{\mathfrak{s}}x_E  \\ 0 &  \displaystyle \sum_{E \in \mathcal{A}(L)^{II}} c_E^{\mathfrak{s}}x_E - \sum_{E \in \mathcal{A}(L)^{I}} c_E^{\mathfrak{s}}x_E \end{pmatrix*}$}.\]

If $x^{\mathfrak{s}}, a, b$ satisfy property (B) then 
\[ \sum_{E \in \mathcal{A}(L)^{I}} c_E^{\mathfrak{s}}x_E = \sum_{E \in \mathcal{A}(L)^{III}} c_E^{\mathfrak{s}}x_E = \sum_{E \in \mathcal{A}(L)^{III}} c_{E}^{\mathfrak{s}}x_{(E \setminus \{a\}) \cup \{b\}} = 0 \] and hence $\partial_{a,b}x^{\mathfrak{s}} $ is diagonal:

\[\partial_{a,b}x^{\mathfrak{s}} =\scalebox{.7}{ $\begin{pmatrix*}[c]   x^{\mathfrak{s}}  & 0 \\  0  &   \displaystyle \sum_{E \in \mathcal{A}(L)^{II}} c_E^{\mathfrak{s}}x_E + \sum_{E \in \mathcal{A}(L)^{IV}} c_{E}^{\mathfrak{s}}x_{(E \setminus \{b\}) \cup \{a\}} \end{pmatrix*}$}.\]

%
%
%
%In the latter case, by the upset property we also have that \[ \sum_{\substack{k \in C_v \\ k \prec a}} \sum_{E : k \in E} c_{E}^{\mathfrak{s}}x_{E} =0\] and for all $E$ such that $a \in E$, $c_E^\mathfrak{s} = 0$. Therefore,
%
%
%\[\partial_{a,b}x^{\mathfrak{s}} =\scalebox{.7}{ $\begin{pmatrix*}[c]   \displaystyle \sum_{E : b \in E} c_E^{\mathfrak{s}}x_E +   \sum_{\substack{k \in C_v \\ k \nprec a,b}} \sum_{E : k \in E} c_{E}^{\mathfrak{s}}x_{E} + \sum_{E \cap C_v = \emptyset} c_E^{\mathfrak{s}}x_E  & 0 \\  0  &   \displaystyle \sum_{E : b \in E} c_{E}^{\mathfrak{s}}x_{(E \setminus \{b\}) \cup \{a\}} + \sum_{\substack{k \in C_v \\ k \nprec a,b}} \sum_{E : k \in E} c_{E}^{\mathfrak{s}}x_{E} + \sum_{E \cap C_v = \emptyset} c_E^{\mathfrak{s}}x_E \end{pmatrix*}$}.\]
This also shows that there is a real matrix $S'$ such that $S' (\partial_{a,b} U) (S')^{-1}$ is upper triangular. Consequently, $S'(\partial S) M^{T_{P'}} (S'(\partial S))^{-1}$ is upper triangular. %, which means that the matrices $G_{E}'$ such that $M^{T_{P'}} = \sum x_{E_k} G_{E}'$ are simultaneously upper-triangularizable.

Finally, notice that $R_{P'} \subset R_{P}$ and if $(a',b') \in R_{P'}$ then $\{a',b'\} \cup \{a,b\} =\emptyset$ and either $a', b' \prec a$ or $a', b' \npreceq b$. So, by inspection, the eigenvalues of $M^{T_{P'}}$ together when any pair from $R_{P'}$ satisfy one of the conditions $(A)$ and $(B)$ from the definition of the upset property.

\end{proofof}

%***********************************************************
%****************************************************************** 
 \bibliographystyle{plain}
% \bibliography{MarkovChain}

\begin{thebibliography}{10}

\bibitem{ayyer2014combinatorial}
Arvind Ayyer, Steven Klee, and Anne Schilling.
\newblock Combinatorial {M}arkov chains on linear extensions.
\newblock {\em Journal of Algebraic Combinatorics}, 39(4):853--881, 2014.

\bibitem{ayyer2015markov}
Arvind Ayyer, Anne Schilling, Benjamin Steinberg, and Nicolas~M Thi{\'e}ry.
\newblock Markov chains, $\mathcal{R}$-trivial monoids and representation
  theory.
\newblock {\em International Journal of Algebra and Computation},
  25(01n02):169--231, 2015.

\bibitem{berg2011primitive}
Chris Berg, Nantel Bergeron, Sandeep Bhargava, and Franco Saliola.
\newblock Primitive orthogonal idempotents for $\mathcal{R}$-trivial monoids.
\newblock {\em Journal of Algebra}, 348(1):446--461, 2011.

\bibitem{bidigare1999combinatorial}
Pat Bidigare, Phil Hanlon, and Dan Rockmore.
\newblock A combinatorial description of the spectrum for the {T}setlin library
  and its generalization to hyperplane arrangements.
\newblock {\em Duke Mathematical Journal}, 99(1):135--174, 1999.

\bibitem{bjorner2008random}
Anders Bj{\"o}rner.
\newblock Random walks, arrangements, cell complexes, greedoids, and
  self-organizing libraries.
\newblock In {\em Building Bridges}, pages 165--203. Springer, 2008.

\bibitem{bjorner2009note}
Anders Bj{\"o}rner.
\newblock Note: Random-to-front shuffles on trees.
\newblock {\em Electronic Communications in Probability}, 14:36--41, 2009.

\bibitem{brown2000semigroups}
Kenneth~S. Brown.
\newblock Semigroups, rings, and {M}arkov chains.
\newblock {\em Journal of Theoretical Probability}, 13(3):871--938, 2000.

\bibitem{brown1998random}
Kenneth~S Brown and Persi Diaconis.
\newblock Random walks and hyperplane arrangements.
\newblock {\em Annals of Probability}, pages 1813--1854, 1998.

\bibitem{donnelly1991heaps}
Peter Donnelly.
\newblock The heaps process, libraries, and size-biased permutations.
\newblock {\em Journal of Applied Probability}, pages 321--335, 1991.

\bibitem{green1951structure}
James~A Green.
\newblock On the structure of semigroups.
\newblock {\em Annals of Mathematics}, pages 163--172, 1951.

\bibitem{haiman1992dual}
Mark~D Haiman.
\newblock Dual equivalence with applications, including a conjecture of
  {P}roctor.
\newblock {\em Discrete Mathematics}, 99(1):79--113, 1992.

\bibitem{hendricks1972stationary}
W.J. Hendricks.
\newblock The stationary distribution of an interesting {M}arkov chain.
\newblock {\em Journal of Applied Probability}, pages 231--233, 1972.

\bibitem{hendricks1973extension}
W.J. Hendricks.
\newblock An extension of a theorem concerning an interesting {M}arkov chain.
\newblock {\em Journal of Applied Probability}, pages 886--890, 1973.

\bibitem{kapoor1991stochastic}
Sanjiv Kapoor and Edward~M. Reingold.
\newblock Stochastic rearrangement rules for self-organizing data structures.
\newblock {\em Algorithmica}, 6(1-6):278--291, 1991.

\bibitem{malvenuto1994evacuation}
Claudia Malvenuto and Christophe Reutenauer.
\newblock Evacuation of labelled graphs.
\newblock {\em Discrete Mathematics}, 132(1):137--143, 1994.

\bibitem{phatarfod1991matrix}
Ravindra~M. Phatarfod.
\newblock On the matrix occurring in a linear search problem.
\newblock {\em Journal of Applied Probability}, pages 336--346, 1991.

\bibitem{poznanovic2017properties}
Svetlana Poznanovi{\'c} and Kara Stasikelis.
\newblock Properties of the promotion {M}arkov chain on linear extensions.
\newblock {\em Journal of Algebraic Combinatorics}, pages 1--24, 2017.

\bibitem{Rhodes:2017rt}
John Rhodes and Anne Schilling.
\newblock Unified theory for finite {M}arkov chains.
\newblock {\em arXiv:1711.10689}, 2017.

\bibitem{schocker2008radical}
M.~Schocker.
\newblock Radical of weakly ordered semigroup algebras.
\newblock {\em Journal of Algebraic Combinatorics}, 28(1):231--234, Aug 2008.

\bibitem{schutzenberger1972promotion}
Marcel-Paul Sch{\"u}tzenberger.
\newblock Promotion des morphismes d'ensembles ordonn{\'e}s.
\newblock {\em Discrete Mathematics}, 2(1):73--94, 1972.

\bibitem{sta97}
Richard~P Stanley.
\newblock Enumerative combinatorics. {V}ol. 1, vol. 49 of {C}ambridge {S}tudies
  in {A}dvanced {M}athematics, 1997.

\bibitem{steinberg2006mobius}
Benjamin Steinberg.
\newblock M{\"o}bius functions and semigroup representation theory.
\newblock {\em Journal of Combinatorial Theory, Series A}, 113(5):866--881,
  2006.

\bibitem{steinberg2008mobius}
Benjamin Steinberg.
\newblock M{\"o}bius functions and semigroup representation theory {II}:
  {C}haracter formulas and multiplicities.
\newblock {\em Advances in Mathematics}, 217(4):1521--1557, 2008.

\end{thebibliography}

\end{document}